\documentclass[11pt,a4paper]{article}

\usepackage[margin=27mm]{geometry}
\usepackage{amsmath,amssymb,amsthm,mathtools,mathrsfs}
\usepackage{microtype}
\usepackage{needspace}
\usepackage[hidelinks,unicode]{hyperref}
\usepackage{aliascnt}
\usepackage[nameinlink,noabbrev]{cleveref}
\usepackage{indentfirst}

\hypersetup{
  pdftitle={Every Compact Metric Space Is Isometrically Embeddable into the Gromov--Hausdorff Space}
}

\allowdisplaybreaks[2]

\theoremstyle{plain}
\newtheorem{theorem}{Theorem}[section]

\newaliascnt{lemma}{theorem}
\newtheorem{lemma}[lemma]{Lemma}
\aliascntresetthe{lemma}

\newaliascnt{proposition}{theorem}
\newtheorem{proposition}[proposition]{Proposition}
\aliascntresetthe{proposition}

\newaliascnt{corollary}{theorem}
\newtheorem{corollary}[corollary]{Corollary}
\aliascntresetthe{corollary}

\theoremstyle{definition}
\newaliascnt{definition}{theorem}
\newtheorem{definition}[definition]{Definition}
\aliascntresetthe{definition}

\newaliascnt{problem}{theorem}
\newtheorem{problem}[problem]{Problem}
\aliascntresetthe{problem}

\crefname{theorem}{Theorem}{Theorems}
\crefname{lemma}{Lemma}{Lemmas}
\crefname{proposition}{Proposition}{Propositions}
\crefname{corollary}{Corollary}{Corollaries}
\crefname{definition}{Definition}{Definitions}
\crefname{problem}{Problem}{Problems}

\newcommand{\M}{\mathcal{M}}
\newcommand{\dGH}{d_{\mathrm{GH}}}
\newcommand{\dis}{\operatorname{dis}}
\newcommand{\diam}{\operatorname{diam}}
\newcommand{\dist}{\operatorname{dist}}
\newcommand{\Corr}{\operatorname{Corr}}
\newcommand{\Isom}{\operatorname{Isom}}
\newcommand{\Aut}{\operatorname{Aut}}
\newcommand{\cL}{\mathcal L}
\newcommand{\id}{\operatorname{id}}

\title{Every Compact Metric Space Is Isometrically Embeddable into the Gromov--Hausdorff Space}
\author{Ryo Fukano}
\date{November 04, 2026}

\begin{document}
\maketitle
\begin{NoHyper}
\begingroup
\renewcommand{\thefootnote}{}
\endgroup
\end{NoHyper}

\begin{abstract}
Let $(\M,\dGH)$ denote the Gromov--Hausdorff space. 
We prove that every nonempty compact metric space is isometrically embeddable into $(\M, \dGH)$, answering a question of Iliadis. 
The corresponding result was previously known for finite metric spaces and several larger classes. 
However, the general compact case remained open. 
More precisely, let $D>0$ and let $K$ be a nonempty compact metric space with diameter at most $D$. 
On a fixed Cantor space, we associate a metric to each $1$-Lipschitz function from $K$ to $[0,D]$ so that the Gromov--Hausdorff distance between any two resulting spaces is exactly the uniform distance between the corresponding functions. 
The resulting spaces have diameter at most $76D$. As a consequence, the subspace of $\M$ represented by Cantor spaces is itself isometrically universal for compact metric spaces. 
We also construct finite approximations of these spaces, obtain an exact formula for the Gromov--Hausdorff distance between the finite models, and derive uniform approximation and covering estimates.
\end{abstract}

\section{Introduction}

Let $\M$ denote the space of isometry classes of nonempty compact metric spaces equipped with the Gromov--Hausdorff distance. The question of whether every nonempty compact metric space can be isometrically embedded into $\M$ was posed by S.~Iliadis and recorded in the work of Iliadis--Ivanov--Tuzhilin \cite{IIT}. They proved that every finite metric space admits such an embedding.

Byakuno restated the question in \cite[Problem~1.1]{Byakuno} and proved an isometric embedding theorem for countable products with summable diameter bounds. More precisely, if $(r_n)_{n\ge1}$ is a summable sequence of positive numbers and, for each $n$, one takes the subspace of $\M$ consisting of isometry classes of compact metric spaces of diameter at most $r_n$, then the countable $\ell^\infty$-product of these subspaces admits an isometric embedding into $\M$ \cite[Theorem~1.3]{Byakuno}. This gives another class of spaces that can be isometrically embedded into $\M$, but does not settle the compact metric space problem in general.

Two related universality results concern different target spaces. Ivanov--Tuzhilin proved that every bounded metric space is isometrically embeddable into the Gromov--Hausdorff class in which noncompact metric spaces are allowed \cite{ITclass}. Since that target is larger than $\M$, however, their result alone does not imply that every image point has a compact representative. Ishiki--Koshino proved that, on any fixed uncountable compact metrizable space, the space of all metrics inducing the given topology, equipped with the uniform metric, is isometrically universal for the class of compact metric spaces \cite{IK}. Nevertheless, when one passes from the uniform metric to the Gromov--Hausdorff metric, changing the correspondence between points may decrease the distance. Thus, their result does not directly imply our main theorem.

Very recently, Ishiki proved that $\M$ is homeomorphic to the separable Hilbert space $\ell^2$ \cite{IshikiTopology}. Since every separable metrizable space admits a topological embedding into this Hilbert space, his result shows that $\M$ is topologically universal for separable metrizable spaces. Our main theorem complements this topological universality by establishing metric universality for compact metric spaces.

\Needspace{6\baselineskip}
\begin{theorem}[Main theorem]\label{thm:main}
For every nonempty compact metric space $(K,d_K)$, there exists an isometric embedding
\[
 \Phi:(K,d_K)\longrightarrow(\M,\dGH).
\]
Moreover, if $D:=\diam K>0$, then one may choose a Cantor space $\Omega$ and a family of metrics $(\rho_x)_{x\in K}$ on $\Omega$, each inducing the given topology on $\Omega$, such that
\[
 \Phi(x)=[(\Omega,\rho_x)]
 \qquad(x\in K).
\]
For all $x,y\in K$, these metrics satisfy
\[
 \diam_{\rho_x}(\Omega)\le76D,
 \qquad
 \dGH\bigl((\Omega,\rho_x),(\Omega,\rho_y)\bigr)=d_K(x,y).
\]
\end{theorem}

The main theorem follows from the following realization theorem for a Lipschitz function space.

\begin{theorem}\label{thm:lip}
Let $(K,d_K)$ be a nonempty compact metric space, and suppose that $D>0$ and $\diam K\le D$. Define
\[
 \cL_D(K)
 :=\bigl\{f:K\to[0,D]\mid
       |f(x)-f(y)|\le d_K(x,y)\quad(x,y\in K)\bigr\},
\]
and, for every bounded function $h:K\to\mathbb R$, set $\lVert h\rVert_\infty:=\sup_{x\in K}|h(x)|$.
Then there exist a Cantor space $\Omega$, fixed independently of $f$, and a metric $\rho_f$ on $\Omega$ for each $f\in\cL_D(K)$ such that, for all $f,g\in\cL_D(K)$,
\begin{equation}\label{eq:mainlip}
 \dGH\bigl((\Omega,\rho_f),(\Omega,\rho_g)\bigr)
 =\lVert f-g\rVert_\infty
 =\frac12\sup_{\omega,\omega'\in\Omega}
       |\rho_f(\omega,\omega')-\rho_g(\omega,\omega')|.
\end{equation}

All the metrics $\rho_f$ induce the given topology on $\Omega$ and satisfy $\diam_{\rho_f}(\Omega)\le76D$.
In particular, if the identity relation $I_\Omega:=\{(\omega,\omega)\mid \omega\in\Omega\}$ is regarded as a correspondence from $(\Omega,\rho_f)$ to $(\Omega,\rho_g)$, then $\dis I_\Omega=2\dGH\bigl((\Omega,\rho_f),(\Omega,\rho_g)\bigr)$.
\end{theorem}

The compact case is therefore complete. On the other hand, Borzov--Ivanov--Tuzhilin's result implies that $\mathbb R$ is not isometrically embeddable into $\M$ \cite[\S4.4, Corollary~9, p.~657]{BIT}. This suggests the following natural problem.

\begin{problem}\label{prob:characterize}
Characterize the metric spaces that are isometrically embeddable into $(\M,\dGH)$.
\end{problem}

This paper is organized as follows. Section 2 collects the preliminaries on the Gromov--Hausdorff distance and correspondences. Section 3 constructs the hierarchical coding of the given compact metric space and establishes the rigidity properties of correspondences that will be used in the lower-bound argument. In Section 4, we construct a family of metrics on a fixed Cantor space parametrized by Lipschitz functions on the original compact metric space and prove the isometric realization theorem, from which the main theorem follows. Section 5 develops finite approximations of this construction, proves an exact formula for the Gromov--Hausdorff distance between the finite approximating spaces, and derives uniform approximation and covering estimates.

\section{Preliminaries}

Throughout the paper, all metric spaces are assumed to be nonempty unless stated otherwise. A \emph{Cantor space} means a topological space homeomorphic to the usual middle-third Cantor set; no particular metric is prescribed. For a positive integer $n$, write $[n]:=\{1,\dots,n\}$.
Given a metric space $(X,d)$, a point $x\in X$, a real number $r>0$, and nonempty subsets $A,B\subseteq X$, set
\begin{align*}
 B_d(x,r)&:=\{y\in X\mid d(x,y)<r\},\\
 d(x,A)&:=\inf_{a\in A}d(x,a),\\
 \diam_d(A)&:=\sup_{a,a'\in A}d(a,a'),\\
 \dist_d(A,B)&:=\inf_{a\in A,\,b\in B}d(a,b).
\end{align*}
When the ambient metric is clear from the context, the subscript will be omitted. For $\varepsilon>0$, a nonempty subset $A\subseteq X$ is called an $\varepsilon$-net of $X$ if $\sup_{x\in X}d(x,A)\le\varepsilon$.
We further define
\[
 N_X(\varepsilon)
 :=\min\left\{|A|\mid
 \varnothing\ne A\subseteq X\text{ is finite and }
 \sup_{x\in X}d(x,A)\le\varepsilon\right\},
\]
with the convention $N_X(\varepsilon):=\infty$ if no finite $\varepsilon$-net exists. For finite $A$, the net condition is equivalent to requiring, for every $x\in X$, a point $a\in A$ with $d(x,a)\le\varepsilon$.

\begin{definition}\label{def:gh}
For nonempty compact subsets $A,B$ of a metric space $Z$, their Hausdorff distance is defined by
\[
 d_\mathrm{H}^Z(A,B)
 :=\max\left\{
   \sup_{a\in A}d_Z(a,B),
   \sup_{b\in B}d_Z(b,A)
 \right\}.
\]
For nonempty compact metric spaces $X,Y$, their Gromov--Hausdorff distance is
\[
 \dGH(X,Y)
 :=\inf_{(Z,\varphi,\psi)}
 d_\mathrm{H}^Z\bigl(\varphi(X),\psi(Y)\bigr),
\]
where the infimum is taken over all triples $(Z,\varphi,\psi)$ consisting of a metric space $Z$ and isometric embeddings $\varphi:X\to Z$ and $\psi:Y\to Z$. The space $\M$ is the set of isometry classes of nonempty compact metric spaces equipped with this metric. For isometry classes we also write $\dGH([X],[Y]):=\dGH(X,Y)$, which is independent of the choice of representatives.
\end{definition}

\begin{definition}\label{def:corr}
A \emph{correspondence} between sets $X$ and $Y$ is a relation $R\subseteq X\times Y$ such that
\[
 \{x\in X\mid \exists y\in Y,\ (x,y)\in R\}=X,
 \qquad
 \{y\in Y\mid \exists x\in X,\ (x,y)\in R\}=Y.
\]
The set of all correspondences between $X$ and $Y$ is denoted by $\Corr(X,Y)$. For subsets $A\subseteq X$ and $B\subseteq Y$, put
\[
 R[A]:=\{y\in Y\mid \exists x\in A,\ (x,y)\in R\},
 \qquad
 R^{-1}[B]:=\{x\in X\mid \exists y\in B,\ (x,y)\in R\}.
\]
If $(X,d_X)$ and $(Y,d_Y)$ are metric spaces and $R$ is a correspondence between them, the \emph{distortion} of $R$ is defined by
\[
 \dis R
 :=\sup\bigl\{|d_X(x,x')-d_Y(y,y')|\mid
        (x,y),(x',y')\in R\bigr\}.
\]
\end{definition}

\begin{proposition}[{\cite[Theorem~7.3.25]{BBI}}]
\label{prop:ghcorr}
For nonempty compact metric spaces $X,Y$,
\begin{equation}\label{eq:GHformula}
  2\dGH(X,Y)=\inf_{R\in\Corr(X,Y)}\dis R.
\end{equation}
\end{proposition}

\section{Construction of the isometric realization I}

Fix from now on a nonempty compact metric space $(K,d_K)$ and a constant $D>0$ such that $\diam K\le D$. For each $k\ge0$, set $\gamma_k:=D\,4^{-k}$ and $\lambda_k:=8\gamma_k=8D\,4^{-k}$.

\begin{definition}\label{def:covertree}
For a sequence of positive integers $(q_k)_{k\ge0}$, define
\[
 T_0:=\{\varnothing\},
 \qquad
 T_k:=\prod_{j=0}^{k-1}[q_j]\quad(k\ge1).
\]
A family of nonempty compact sets
\[
 (K_v)_{v\in\bigcup_{k\ge0}T_k}
\]
will be called a \emph{labeled covering tree subordinate to $(\gamma_k)$} if each $K_v$ is a subset of $K$, $K_\varnothing=K$, and, for every $k\ge0$, $v\in T_k$, and $i\in[q_k]$,
\[
 \diam(K_v)\le\gamma_k,
 \qquad
 K_v=\bigcup_{j=1}^{q_k}K_{vj},
 \qquad
 \varnothing\ne K_{vi}\subseteq K_v.
\]
Here, for $v=(v_0,\dots,v_{k-1})$, we write $vi:=(v_0,\dots,v_{k-1},i)\in T_{k+1}$.
The underlying rooted tree has vertex set $\bigcup_{k\ge0}T_k$, root $\varnothing$, and edges from $v$ to $vi$. Distinct labels give distinct vertices even when the associated sets coincide; in particular, we allow $K_{vi}=K_{vj}$ for $i\ne j$.
\end{definition}

\begin{proposition}\label{prop:covertree}
There exists a labeled covering tree subordinate to $(\gamma_k)$ such that $q_k\ge7$ for every $k\ge0$.
\end{proposition}

\begin{proof}
Start with $K_\varnothing:=K$, which satisfies $\diam K_\varnothing\le\gamma_0$. Suppose that the family $(K_v)_{v\in T_k}$ at depth $k$ has already been constructed. For each $v\in T_k$, compactness of $K_v$ allows us to choose a positive integer $m(v)$, points $x_{v,\ell}\in K_v$, and radii $0<r_{v,\ell}<\gamma_{k+1}/2$ for $1\le\ell\le m(v)$ such that
\[
 K_v\subseteq\bigcup_{\ell=1}^{m(v)}
 B_{d_K}(x_{v,\ell},r_{v,\ell}).
\]
Set
\[
 L_{v,\ell}
 :=\overline{K_v\cap B_{d_K}(x_{v,\ell},r_{v,\ell})}^{\,K_v}
 \qquad(1\le\ell\le m(v)).
\]
Each $L_{v,\ell}$ is a nonempty compact set, and
\[
 K_v=\bigcup_{\ell=1}^{m(v)}L_{v,\ell},
 \qquad
 \diam(L_{v,\ell})\le2r_{v,\ell}<\gamma_{k+1}.
\]
Since $T_k$ is finite, we may set $q_k:=\max\left\{7,\max_{v\in T_k}m(v)\right\}<\infty$.
For each $v\in T_k$, choose a surjection $\theta_v:[q_k]\longrightarrow[m(v)]$ and define $K_{vi}:=L_{v,\theta_v(i)}$ for $i\in[q_k]$.
The surjectivity of $\theta_v$ ensures that the children cover $K_v$, and each child is nonempty with diameter less than $\gamma_{k+1}$. This completes the inductive construction.
\end{proof}

Fix once and for all a covering tree provided by \cref{prop:covertree}.

\begin{definition}\label{def:boundary}
The boundary, that is, the set of infinite branches of the rooted tree from \cref{def:covertree}, is $S:=\prod_{k=0}^{\infty}[q_k]$, with each factor given the discrete topology and $S$ given the product topology. For $\alpha,\beta,\xi,\dots\in S$, the symbols $\alpha_j,\beta_j,\xi_j,\dots$ denote the corresponding $j$th coordinates. In particular, for $\alpha=(\alpha_0,\alpha_1,\dots)\in S$, set $\alpha|_0:=\varnothing$, and for $k\ge1$ write $\alpha|_k:=(\alpha_0,\dots,\alpha_{k-1})\in T_k$.
For $v\in T_k$, define the \emph{cylinder set} $C_v:=\{\alpha\in S\mid \alpha|_k=v\}$.
Then, for every $k\ge0$, $v\in T_k$, and $\alpha\in S$,
\begin{equation}\label{eq:cylinder-identities}
 C_\varnothing=S,
 \qquad
 C_v=\bigsqcup_{i=1}^{q_k}C_{vi},
 \qquad
 \bigcap_{m=0}^{\infty}C_{\alpha|_m}=\{\alpha\}.
\end{equation}
For distinct $\alpha,\beta\in S$, define their first disagreement level by $\ell(\alpha,\beta):=\min\{k\ge0\mid \alpha_k\ne\beta_k\}$.
\end{definition}

\begin{proposition}\label{prop:S-cantor}
With the product topology, $S$ is homeomorphic to the Cantor space.
\end{proposition}

\begin{proof}
Each $[q_k]$ is a finite discrete space containing at least two points. Hence $S$ is nonempty, compact, metrizable, zero-dimensional, and has no isolated points. Brouwer's characterization of Cantor space \cite[Theorem~7.4]{Kechris} therefore applies.
\end{proof}

\begin{proposition}\label{prop:pi}
For each $\alpha\in S$, define $\pi(\alpha)$ by
\begin{equation}\label{eq:pi}
 \{\pi(\alpha)\}:=\bigcap_{k=0}^{\infty}K_{\alpha|_k}.
\end{equation}
Then $\pi:S\to K$ is a continuous surjection. Moreover, if $\alpha,\beta\in S$ are distinct and $k:=\ell(\alpha,\beta)$, then $d_K(\pi(\alpha),\pi(\beta))\le\gamma_k$.
\end{proposition}

\begin{proof}
For each $\alpha\in S$, the sequence $(K_{\alpha|_k})_{k\ge0}$ is decreasing and consists of nonempty compact sets, while $\diam K_{\alpha|_k}\le\gamma_k\to0$. Hence their intersection consists of exactly one point, so \eqref{eq:pi} defines a map $\pi:S\to K$.

Take any $z\in K$. Since $K_\varnothing=K$ and $K_v=\bigcup_{i=1}^{q_k}K_{vi}$, one can inductively choose a sequence $(i_k)_{k\ge0}$ such that
\[
 i_k\in[q_k],\qquad
 z\in K_{(i_0,\dots,i_k)}
 \quad(k\ge0).
\]
If $\alpha:=(i_k)_{k\ge0}\in S$, then $z\in K_{\alpha|_k}$ for every $k$, and therefore \eqref{eq:pi} gives $\pi(\alpha)=z$. Thus $\pi$ is surjective. If $\alpha,\beta\in S$ are distinct and $k:=\ell(\alpha,\beta)$, then $\pi(\alpha),\pi(\beta)\in K_{\alpha|_k}=K_{\beta|_k}$, which yields the asserted estimate. Finally, given $\varepsilon>0$, choose $k$ such that $\gamma_k<\varepsilon$. If $\beta\in C_{\alpha|_k}$, then
$d_K(\pi(\alpha),\pi(\beta))\le\gamma_k<\varepsilon$. Hence $\pi$ is continuous.
\end{proof}

\begin{definition}\label{def:rigid}
Write $\Isom(X)$ for the group of bijective self-isometries of a metric space $X$. The space $X$ is called \emph{rigid} if its only bijective self-isometry is the identity; equivalently, $\Isom(X)=\{\id_X\}$.
\end{definition}

\begin{lemma}\label{lem:rigidmetric}
For every integer $q\ge7$, there exists a rigid metric $a^{(q)}$ on $[q]$ such that
\[
 a^{(q)}(i,i)=0,
 \qquad
 a^{(q)}(i,j)\in\{2,3\}\quad(i\ne j).
\]
\end{lemma}

\begin{proof}
For a graph $G$, let $\Aut(G)$ denote its automorphism group. Let $G_q$ be the tree obtained by attaching to a single central vertex three paths having respectively $1$, $2$, and $q-4$ edges, with pairwise disjoint sets of vertices away from the center. The number of vertices is $1+1+2+(q-4)=q$.
Since $q\ge7$, the three path lengths are distinct, and the center is the unique vertex of degree $3$. Hence every automorphism fixes the center and maps each branch to itself. Moreover, each vertex on a branch is uniquely determined by its graph distance from the center. Thus $\Aut(G_q)=\{\id_{G_q}\}$.

Identify $[q]$ with the vertex set of $G_q$, and define, for $i,j\in[q]$,
\[
 a^{(q)}(i,j):=
 \begin{cases}
  0,&i=j,\\
  2,&i\ne j\text{ and }\{i,j\}\text{ is an edge of }G_q,\\
  3,&i\ne j\text{ and }\{i,j\}\text{ is not an edge of }G_q.
 \end{cases}
\]
For three distinct points, every distance is either $2$ or $3$, so the triangle inequality follows from $a^{(q)}(i,j)\le3<4\le a^{(q)}(i,k)+a^{(q)}(k,j)$. The off-diagonal pairs at distance $2$ are exactly the edges of $G_q$. Therefore $\Isom([q],a^{(q)})=\Aut(G_q)=\{\id_{[q]}\}$.
Hence $([q],a^{(q)})$ is rigid.
\end{proof}

\begin{lemma}\label{lem:nondecreasing}
Let $(X,d)$ be a finite metric space, and let $\sigma:X\to X$ be a bijection satisfying
\[
 d(x,y)\le d(\sigma(x),\sigma(y))
 \qquad(x,y\in X).
\]
Then $\sigma$ is an isometry. Consequently, if $X$ is rigid, then $\sigma=\id_X$.
\end{lemma}

\begin{proof}
The inverse map $\sigma^{-1}:X\to X$ is a surjective nonexpanding map. Since $X$ is finite and hence compact, \cite[Theorem~1.6.15(1)]{BBI} implies that $\sigma^{-1}$ is an isometry. Therefore $\sigma$ is an isometry, and the final assertion follows from rigidity.
\end{proof}

For each $k\ge0$, fix one of the rigid metrics $a^{(q_k)}$ supplied by \cref{lem:rigidmetric}.

\begin{definition}\label{def:bmetric}
For $\alpha,\beta\in S$, set $b(\alpha,\beta):=0$ if $\alpha=\beta$. If $\alpha\ne\beta$, let $k:=\ell(\alpha,\beta)$ and define
\begin{equation}\label{eq:bmetric}
 b(\alpha,\beta)
 :=\lambda_k a^{(q_k)}(\alpha_k,\beta_k).
\end{equation}
\end{definition}

\begin{proposition}\label{prop:compactindex}
The function $b$ is a metric on $S$, and the topology induced by $b$ coincides with the product topology. Thus $(S,b)$ is a compact metric space homeomorphic to the Cantor space. Moreover, for every $k\ge0$, $v\in T_k$, and distinct $i,j\in[q_k]$,
\begin{align}
 \diam_b(C_v)&\le3\lambda_k,\label{eq:cylbounds}\\
 b(\alpha,\beta)&=\lambda_k a^{(q_k)}(i,j)
   &&(\alpha\in C_{vi},\ \beta\in C_{vj}),\label{eq:childcross}\\
 d_K(\pi(\alpha),\pi(\beta))&\le\frac1{16}b(\alpha,\beta)
   &&(\alpha,\beta\in S),\label{eq:pilip}\\
 \diam_b(S)&\le24D.\label{eq:bdiam}
\end{align}
\end{proposition}

\begin{proof}
Symmetry and definiteness follow immediately from the definition. To verify the triangle inequality, take three distinct points $\alpha,\beta,\xi\in S$ and put
\[
 k:=\min\{j\ge0\mid |\{\alpha_j,\beta_j,\xi_j\}|\ge2\}.
\]
If $|\{\alpha_k,\beta_k,\xi_k\}|=3$, then the first disagreement level of each of the three pairs is $k$. Hence, for example,
\[
 b(\alpha,\beta)
 =\lambda_k a^{(q_k)}(\alpha_k,\beta_k)
 \le\lambda_k\bigl(
 a^{(q_k)}(\alpha_k,\xi_k)
 +a^{(q_k)}(\xi_k,\beta_k)\bigr)
 =b(\alpha,\xi)+b(\xi,\beta).
\]
Cyclically permuting $\alpha,\beta,\xi$ gives the other two triangle inequalities.

Suppose instead that $|\{\alpha_k,\beta_k,\xi_k\}|=2$. After relabelling the points if necessary, we may assume $\alpha_k=\beta_k=:i$ and $\xi_k=:j\ne i$.
Then
\[
 b(\alpha,\xi)=b(\beta,\xi)
 =\lambda_k a^{(q_k)}(i,j)\ge2\lambda_k,
 \qquad
 b(\alpha,\beta)\le3\lambda_{k+1}=\frac34\lambda_k.
\]
Therefore
\[
 b(\alpha,\beta)
 \le b(\alpha,\xi)+b(\xi,\beta),
\]
and the remaining two triangle inequalities follow from
$b(\alpha,\xi)=b(\beta,\xi)$ and $b(\alpha,\beta)\ge0$. Thus $b$ is a metric.

\eqref{eq:cylbounds} and \eqref{eq:childcross} follow directly from \eqref{eq:bmetric}. Moreover, for $k\ge1$ and $v\in T_k$, $\dist_b(C_v,S\setminus C_v)\ge2\lambda_{k-1}>0$, so every cylinder is both open and closed in the $b$-topology. Conversely, given $\alpha\in S$ and $r>0$, choose $k$ such that $3\lambda_k<r$. Then $C_{\alpha|_k}\subseteq B_b(\alpha,r)$.
Since the cylinder sets form a basis for the product topology and, by the two preceding observations, also form a basis for the $b$-topology, the two topologies coincide.

For distinct $\alpha,\beta\in S$ with $k:=\ell(\alpha,\beta)$, \cref{prop:pi} and the inequality $a^{(q_k)}(\alpha_k,\beta_k)\ge2$ give
\[
 b(\alpha,\beta)\ge2\lambda_k=16\gamma_k,
 \qquad
 d_K(\pi(\alpha),\pi(\beta))\le\gamma_k.
\]
This implies \eqref{eq:pilip}; when $\alpha=\beta$, both sides of \eqref{eq:pilip} vanish. Finally, applying \eqref{eq:cylbounds} to $v=\varnothing$ yields $\diam_b(S)\le3\lambda_0=24D$.
\end{proof}

\begin{lemma}\label{lem:cylinder-rigidity}
Let $k\ge0$ and $v\in T_k$, and suppose that $E\subseteq C_v$ satisfies
\[
 E_i:=E\cap C_{vi}\ne\varnothing
 \qquad(i\in[q_k]).
\]
If $P,Q\in\Corr(E,E)$ satisfy
\begin{equation}\label{eq:onesided}
 \sup_{\substack{(s,t)\in P\\(u,w)\in Q}}
 \bigl(b(s,u)-b(t,w)\bigr)<\frac{\lambda_k}{2},
\end{equation}
then, for every $i\in[q_k]$,
\begin{align*}
 P[E_i]&=E_i=P^{-1}[E_i],\\
 Q[E_i]&=E_i=Q^{-1}[E_i].
\end{align*}
Consequently,
\[
 P\cap(E_i\times E_i),\quad Q\cap(E_i\times E_i)
 \in\Corr(E_i,E_i).
\]
\end{lemma}

\begin{proof}

Set $\lambda:=\lambda_k$. By \eqref{eq:cylinder-identities}, $E=\bigsqcup_{i=1}^{q_k}E_i$, and \cref{prop:compactindex} implies that, for any $i\ne j$, $x,x'\in E_i$, and $y\in E_j$, we have $b(x,x')\le\frac34\lambda$ and $b(x,y)=\lambda a^{(q_k)}(i,j)\ge2\lambda$.

Fix $j\in[q_k]$. Since the second projection of $Q$ is surjective, choose $(u_0,w_0)\in Q$ with $w_0\in E_j$. Suppose also that $(s,t)\in P$ with $t\in E_j$. Write $s\in E_i$ and $u_0\in E_{i'}$. If $i\ne i'$, then these bounds give $b(s,u_0)-b(t,w_0)\ge2\lambda-\frac34\lambda=\frac54\lambda$, contradicting \eqref{eq:onesided}. Hence $P^{-1}[E_j]$ is contained in a single $E_i$.

Since the second projection of $P$ is also surjective, choose $(s_0,t_0)\in P$ with $t_0\in E_j$. For any $(u,w)\in Q$ with $w\in E_j$, write $s_0\in E_i$ and $u\in E_{i'}$. If $i\ne i'$, then the same bounds give $b(s_0,u)-b(t_0,w)\ge2\lambda-\frac34\lambda=\frac54\lambda$, which contradicts \eqref{eq:onesided}. Thus there exists a unique $\psi(j)\in[q_k]$ such that $P^{-1}[E_j]\cup Q^{-1}[E_j]\subseteq E_{\psi(j)}$.

For arbitrary $i\in[q_k]$ and $s\in E_i$, surjectivity of the first projection of $P$ gives $(s,t)\in P$ for some $t\in E$. If $t\in E_j$, then this containment and disjointness of the partition imply $i=\psi(j)$. Therefore $\psi:[q_k]\to[q_k]$ is surjective and hence, since the set is finite, bijective. Put $\sigma:=\psi^{-1}$. Then $P[E_i]\cup Q[E_i]\subseteq E_{\sigma(i)}$ for every $i\in[q_k]$.

Let $H\in\{P,Q\}$. If $y\in E_{\sigma(i)}$, choose $(x,y)\in H$ and write $x\in E_j$. By this inclusion, $y\in E_{\sigma(j)}$, so disjointness of the partition and injectivity of $\sigma$ give $j=i$. Together with the same inclusion, this proves

\[
H[E_i]=E_{\sigma(i)}.
\]

The defining containment for $\psi$ also gives $H^{-1}[E_{\sigma(i)}]\subseteq E_i$. Conversely, every $x\in E_i$ has a partner under $H$, and every such partner belongs to $E_{\sigma(i)}$ by the inclusion above. Hence $E_i\subseteq H^{-1}[E_{\sigma(i)}]$. Thus, for every $i\in[q_k]$, $P[E_i]=E_{\sigma(i)}=Q[E_i]$ and $P^{-1}[E_{\sigma(i)}]=E_i=Q^{-1}[E_{\sigma(i)}]$.

For distinct $i,j\in[q_k]$, use surjectivity of the first projections of $P$ and $Q$ to choose

\[
(s,t)\in P,\quad s\in E_i,
\qquad
(u,w)\in Q,\quad u\in E_j.
\]

By these image identities and the bounds above,

\[
\lambda\Bigl(
a^{(q_k)}(i,j)-a^{(q_k)}(\sigma(i),\sigma(j))
\Bigr)<\frac{\lambda}{2}.
\]

The expression in parentheses is an integer; hence

\[
a^{(q_k)}(i,j)
\le a^{(q_k)}(\sigma(i),\sigma(j)).
\]

For $i=j$ both sides are $0$, so the inequality holds for all $i,j\in[q_k]$. By \cref{lem:nondecreasing} and the rigidity of $([q_k],a^{(q_k)})$, we obtain $\sigma=\id_{[q_k]}$. Substituting this into the image identities proves the claim.

\end{proof}

\section{Construction of the isometric realization II}

\begin{definition}\label{def:Yf}
Set
\[
 \Omega:=S\times\{+,-\},
 \qquad
 \Omega^\varepsilon:=S\times\{\varepsilon\}
 \quad(\varepsilon\in\{+,-\}),
\]
and equip $\Omega$ with the product topology. Define the involution on $\{+,-\}$ by
\[
 \varepsilon^\ast:=
 \begin{cases}
  -,&\varepsilon=+,\\
  +,&\varepsilon=-.
 \end{cases}
\]
For any function $h:K\to\mathbb R$, write $\widehat h:=h\circ\pi$. For $f\in\cL_D(K)$, define $\rho_f$ on $\Omega$ by
\begin{equation}\label{eq:rhof}
 \rho_f\bigl((s,\varepsilon),(t,\varepsilon')\bigr)
 :=
 \begin{cases}
  2b(s,t),&\varepsilon=\varepsilon',\\
  50D+b(s,t)+\widehat f(s)+\widehat f(t),&\varepsilon\ne\varepsilon',
 \end{cases}
\end{equation}
for $s,t\in S$ and $\varepsilon,\varepsilon'\in\{+,-\}$. For a nonempty subset $E\subseteq S$, write
\[
 Y_{f,E}
 :=\bigl(E\times\{+,-\},
 \rho_f|_{(E\times\{+,-\})^2}\bigr).
\]
In particular, let $Y_f:=Y_{f,S}=(\Omega,\rho_f)$, and call $\Omega^+$ and $\Omega^-$ the two layers of $Y_f$.
\end{definition}

\begin{lemma}\label{lem:Ymetric}
For every $f\in\cL_D(K)$, the function $\rho_f$ is a metric on $\Omega$, and the topology it induces agrees with the product topology from \cref{def:Yf}. In particular, $Y_f$ is a compact metric space homeomorphic to the Cantor space. Moreover,
\begin{align}
 \diam_{\rho_f}(\Omega^\varepsilon)&\le48D
 &&(\varepsilon\in\{+,-\}),\label{eq:layer-diam}\\
 \dist_{\rho_f}(\Omega^+,\Omega^-)&\ge50D,\label{eq:layer-gap}\\
 \diam_{\rho_f}(\Omega)&\le76D.\label{eq:Ydiam}
\end{align}
\end{lemma}

\begin{proof}
By \eqref{eq:pilip} and the $1$-Lipschitz property of $f$, for every $s,t\in S$,
\[
 |\widehat f(s)-\widehat f(t)|
 \le d_K(\pi(s),\pi(t))
 \le\frac1{16}b(s,t).
\]
Hence, for $s,t,u\in S$ and $\varepsilon\in\{+,-\}$,
\begin{align*}
 &\left|
 \rho_f((s,\varepsilon),(u,\varepsilon^\ast))
 -\rho_f((t,\varepsilon),(u,\varepsilon^\ast))
 \right|\\
 &\qquad\le |b(s,u)-b(t,u)|
        +|\widehat f(s)-\widehat f(t)|\\
 &\qquad\le\frac{17}{16}b(s,t)
 \le2b(s,t)
 =\rho_f((s,\varepsilon),(t,\varepsilon)).
\end{align*}
This estimate simultaneously gives the two triangle inequalities
\begin{align*}
 \rho_f((s,\varepsilon),(u,\varepsilon^\ast))
 &\le \rho_f((s,\varepsilon),(t,\varepsilon))
      +\rho_f((t,\varepsilon),(u,\varepsilon^\ast)),\\
 \rho_f((t,\varepsilon),(u,\varepsilon^\ast))
 &\le \rho_f((t,\varepsilon),(s,\varepsilon))
      +\rho_f((s,\varepsilon),(u,\varepsilon^\ast)).
\end{align*}
For the remaining direction,
\[
 \rho_f((s,\varepsilon),(t,\varepsilon))
 \le48D<100D
 \le\rho_f((s,\varepsilon),(u,\varepsilon^\ast))
   +\rho_f((u,\varepsilon^\ast),(t,\varepsilon)).
\]
The triangle inequality within a single layer follows from the triangle inequality for $b$. Moreover, \eqref{eq:rhof} is symmetric, distinct points in the same layer have positive distance by definiteness of $b$, and points in different layers are at distance at least $50D$. Since $D>0$, distinct points always have positive distance. Thus $\rho_f$ is a metric.

The restriction to either layer is isometric to $(S,2b)$, and \eqref{eq:layer-gap} shows that the two layers are separated by a positive distance. Therefore the topology induced by $\rho_f$ agrees with the product topology, and $Y_f$ is compact. Since $\Omega$ is nonempty, compact, metrizable, zero-dimensional, and has no isolated points, it is a Cantor space by the same characterization used in \cref{prop:S-cantor}.

Equation \eqref{eq:layer-diam} follows from \eqref{eq:bdiam}, and \eqref{eq:layer-gap} follows directly from \eqref{eq:rhof}. For two points in different layers, $\rho_f((s,+),(t,-))\le50D+24D+2D=76D$. Together with \eqref{eq:layer-diam}, this yields \eqref{eq:Ydiam}.
\end{proof}

\begin{lemma}\label{lem:upper}
Let $f,g\in\cL_D(K)$ and put $\delta:=\lVert f-g\rVert_\infty$. Then
\begin{equation}\label{eq:upper}
 \sup_{\omega,\omega'\in\Omega}
 |\rho_f(\omega,\omega')-\rho_g(\omega,\omega')|=2\delta,
 \qquad
 \dGH(Y_f,Y_g)\le\delta.
\end{equation}
\end{lemma}

\begin{proof}
Let $h:=f-g$. Distances within a single layer do not depend on $f$, whereas for a pair of points in different layers,
\begin{align*}
 &\rho_f((s,+),(t,-))-\rho_g((s,+),(t,-))\\
 &\qquad=\widehat h(s)+\widehat h(t)
 =h(\pi(s))+h(\pi(t)).
\end{align*}
Therefore
\[
 \sup_{\omega,\omega'\in\Omega}
 |\rho_f(\omega,\omega')-\rho_g(\omega,\omega')|
 \le2\lVert h\rVert_\infty=2\delta.
\]
Since the continuous function $|h|$ attains its maximum on the compact space $K$, there exists $z\in K$ with $|h(z)|=\delta$. Choose $\eta\in S$ such that $\pi(\eta)=z$. Then $|\rho_f((\eta,+),(\eta,-))-\rho_g((\eta,+),(\eta,-))|=2|h(z)|=2\delta$, so the upper bound $2\delta$ is attained. Finally, applying \eqref{eq:GHformula} to the identity correspondence $I_\Omega$ gives $2\dGH(Y_f,Y_g)\le\dis I_\Omega=2\delta$.
\end{proof}

\begin{lemma}\label{lem:layers}
Let $f,g\in\cL_D(K)$ and $\varnothing\ne E\subseteq S$, and let $R$ be a correspondence between $Y_{f,E}$ and $Y_{g,E}$ satisfying $\dis R<2D$.
Put $E^\varepsilon:=E\times\{\varepsilon\}$ and define the layer-swap map
\[
 \tau_E:E\times\{+,-\}\longrightarrow E\times\{+,-\},
 \qquad
 \tau_E(t,\varepsilon):=(t,\varepsilon^\ast),
\]
with $\tau_E^0:=\id_{E\times\{+,-\}}$ and $\tau_E^1:=\tau_E$. Then there exists $\nu\in\{0,1\}$ such that
\begin{equation}\label{eq:normalized-relation}
 \widetilde R
 :=\{(x,\tau_E^\nu(y))\mid (x,y)\in R\}
\end{equation}
is a correspondence satisfying
\begin{equation}\label{eq:layer-preservation}
 \dis\widetilde R=\dis R,
 \qquad
 \widetilde R\subseteq
 (E^+\times E^+)\cup(E^-\times E^-).
\end{equation}
Furthermore,
\[
 \widetilde R_\varepsilon
 :=\{(s,t)\in E^2\mid
 ((s,\varepsilon),(t,\varepsilon))\in\widetilde R\}
 \in\Corr(E,E)
 \qquad(\varepsilon\in\{+,-\}).
\]
\end{lemma}

\begin{proof}
Suppose that two points in the same layer of $Y_{f,E}$ are related to points in different layers of $Y_{g,E}$. Thus, for some $\varepsilon,\vartheta,\vartheta'\in\{+,-\}$ with $\vartheta\ne\vartheta'$ and $s,s',t,t'\in E$, we have $((s,\varepsilon),(t,\vartheta))\in R$ and $((s',\varepsilon),(t',\vartheta'))\in R$.
By \eqref{eq:layer-diam} and \eqref{eq:layer-gap}, $\rho_f((s,\varepsilon),(s',\varepsilon))\le48D$ and $\rho_g((t,\vartheta),(t',\vartheta'))\ge50D$. Hence $\dis R\ge50D-48D=2D$, contrary to the hypothesis. Therefore, for each $\varepsilon\in\{+,-\}$, there is a unique sign $\sigma(\varepsilon)\in\{+,-\}$ such that $R[E^\varepsilon]\subseteq E^{\sigma(\varepsilon)}$.
Since both projections of $R$ are surjective, the map $\sigma:\{+,-\}\longrightarrow\{+,-\}$ is bijective. Thus $\sigma$ is either the identity or the sign swap.

The map $\tau_E$ is a self-isometry of $Y_{g,E}$. Choose $\nu=0$ if $\sigma$ is the identity and $\nu=1$ if $\sigma$ is the sign swap. Then the relation $\widetilde R$ defined by \eqref{eq:normalized-relation} is a correspondence, has the same distortion as $R$, and satisfies \eqref{eq:layer-preservation}. The inclusion in \eqref{eq:layer-preservation}, together with surjectivity of both projections of $\widetilde R$, implies that both projections of each $\widetilde R_\varepsilon$ are surjective.
\end{proof}

\begin{proof}[Proof of \cref{thm:lip}]
By \cref{lem:Ymetric,lem:upper}, it remains to prove the Gromov--Hausdorff lower bound in \eqref{eq:mainlip}.

Let $f,g\in\cL_D(K)$ and put $\delta:=\lVert f-g\rVert_\infty$. If $\delta=0$, the conclusion follows from \cref{lem:upper}; hence assume $\delta>0$. Since the continuous function $|f-g|$ attains its maximum on compact $K$, there is $z_0\in K$ such that $|f(z_0)-g(z_0)|=\delta$. Interchanging $f$ and $g$ if necessary, assume $f(z_0)-g(z_0)=\delta$. Fix $\eta\in S$ satisfying $\pi(\eta)=z_0$.

Suppose, toward a contradiction, that there exists a correspondence $R\in\Corr(\Omega,\Omega)$ between $Y_f$ and $Y_g$ with $\Delta:=\dis R<2\delta$.
Since $0<\delta\le D$, we have $\Delta<2D$. Apply \cref{lem:layers} with $E=S$ and replace $R$ by the normalized relation $\widetilde R$ from \eqref{eq:normalized-relation}. This replacement leaves the distortion $\Delta$ unchanged. From now on, write $R$ again for the normalized correspondence and set
\[
 R_\varepsilon
 :=\{(s,t)\in S^2\mid
 ((s,\varepsilon),(t,\varepsilon))\in R\}
 \in\Corr(S,S)
 \qquad(\varepsilon\in\{+,-\}).
\]

For each $k\ge0$, put $C_k:=C_{\eta|_k}$.
We prove by induction that
\begin{equation}\label{eq:induction}
 R_\varepsilon[C_k]=C_k=R_\varepsilon^{-1}[C_k]
 \qquad(\varepsilon\in\{+,-\}).
\end{equation}
Since $C_0=S$, the case $k=0$ follows from $R_\varepsilon\in\Corr(S,S)$.

Assume that \eqref{eq:induction} holds at depth $k$. Then
\[
 P_k:=R_+\cap(C_k\times C_k),
 \qquad
 Q_k:=R_-\cap(C_k\times C_k)
\]
both belong to $\Corr(C_k,C_k)$. Take arbitrary $(s,t)\in P_k$ and $(u,w)\in Q_k$. The points $\pi(s),\pi(t),\pi(u),\pi(w),z_0$ all lie in $K_{\eta|_k}$, whose diameter is at most $\gamma_k$. Since $f(z_0)-g(z_0)=\delta$ and $f,g$ are $1$-Lipschitz,
\begin{align*}
 \widehat f(s)-\widehat g(t)
 &=f(\pi(s))-f(z_0)+\delta+g(z_0)-g(\pi(t))
 \ge\delta-2\gamma_k,\\
 \widehat f(u)-\widehat g(w)
 &=f(\pi(u))-f(z_0)+\delta+g(z_0)-g(\pi(w))
 \ge\delta-2\gamma_k.
\end{align*}
On the other hand, applying the definition of distortion to the two elements
$((s,+),(t,+))$ and $((u,-),(w,-))$ of $R$, and cancelling the common term $50D$ in the cross-layer distances, gives
\[
 \left|
 b(s,u)-b(t,w)
 +\widehat f(s)+\widehat f(u)
 -\widehat g(t)-\widehat g(w)
 \right|
 \le\Delta.
\]
Combining the upper bound $\Delta$ with the two Lipschitz estimates gives a bound uniform over $P_k\times Q_k$:
\[
 \sup_{\substack{(s,t)\in P_k\\(u,w)\in Q_k}}
 \bigl(b(s,u)-b(t,w)\bigr)
 \le\Delta-2\delta+4\gamma_k
 <4\gamma_k=\frac{\lambda_k}{2}.
\]
Moreover, $C_k=C_{\eta|_k}$ meets every child cylinder, so \cref{lem:cylinder-rigidity} applies to $P_k,Q_k$ with $E=C_k$ and $v=\eta|_k$.
Since $C_{k+1}=C_{\eta|_{k+1}}=C_{(\eta|_k)\,\eta_k}$, applying the lemma to the child label $i=\eta_k$ yields
\[
 P_k[C_{k+1}]=C_{k+1}=P_k^{-1}[C_{k+1}],\qquad
 Q_k[C_{k+1}]=C_{k+1}=Q_k^{-1}[C_{k+1}].
\]
The induction hypothesis implies
\[
 R_\varepsilon\cap
 \bigl((C_k\times S)\cup(S\times C_k)\bigr)
 \subseteq C_k\times C_k.
\]
Hence the preceding equalities for $P_k,Q_k$ imply
\[
 R_\varepsilon[C_{k+1}]=C_{k+1}
 =R_\varepsilon^{-1}[C_{k+1}]
 \qquad(\varepsilon\in\{+,-\}),
\]
proving \eqref{eq:induction} at depth $k+1$.

Thus \eqref{eq:induction} holds for every $k\ge0$. Fix $\varepsilon\in\{+,-\}$. Since the first projection of $R_\varepsilon$ is surjective, choose $t_\varepsilon\in S$ with $(\eta,t_\varepsilon)\in R_\varepsilon$. For every $k$, the inclusion $\eta\in C_k$ and \eqref{eq:induction} give $t_\varepsilon\in C_k$, and hence
\[
 t_\varepsilon\in\bigcap_{k=0}^{\infty}C_k=\{\eta\}.
\]
Therefore $(\eta,\eta)\in R_+\cap R_-$, so $R$ contains
\[
 ((\eta,+),(\eta,+)),\qquad ((\eta,-),(\eta,-)).
\]
The distance discrepancy determined by these two elements is
\[
 \left|
 \rho_f((\eta,+),(\eta,-))
 -\rho_g((\eta,+),(\eta,-))
 \right|
 =2|f(z_0)-g(z_0)|
 =2\delta,
\]
contradicting the assumption $\Delta<2\delta$.

Hence every correspondence $R\in\Corr(\Omega,\Omega)$ between $Y_f$ and $Y_g$ satisfies $\dis R\ge2\delta$. By \cref{prop:ghcorr} and \cref{lem:upper}, $\dGH(Y_f,Y_g)=\delta=\lVert f-g\rVert_\infty$. The assertion concerning the identity correspondence also follows from \eqref{eq:upper}.
\end{proof}

\begin{proof}[Proof of \cref{thm:main}]
If $K$ has one point, map it to the isometry class of the one-point metric space. Otherwise, set $D:=\diam K>0$ and, for each $x\in K$, define $f_x(z):=d_K(x,z)$ for $z\in K$. Then $f_x\in\cL_D(K)$. For every $z\in K$, $|f_x(z)-f_y(z)|=|d_K(x,z)-d_K(y,z)|\le d_K(x,y)$, whereas taking $z=y$ gives $|f_x(y)-f_y(y)|=d_K(x,y)$. Hence $\lVert f_x-f_y\rVert_\infty=d_K(x,y)$. Therefore the construction in \cref{thm:lip} shows that $x\longmapsto[(\Omega,\rho_{f_x})]$ is the required isometric embedding, with $\rho_x:=\rho_{f_x}$.
\end{proof}

\begin{corollary}\label{cor:totally-bounded}
Every totally bounded metric space is isometrically embeddable into $\M$.
\end{corollary}

\begin{proof}
The completion $\overline X$ of a totally bounded metric space $X$ is compact. Apply \cref{thm:main} to $\overline X$ and restrict the resulting isometric embedding to $X$.
\end{proof}

\begin{corollary}\label{cor:cantor-universal}
Define $\M_{\mathrm{Cantor}}:=\{[X]\in\M\mid \text{$X$ is homeomorphic to the Cantor space}\}$, and equip $\M_{\mathrm{Cantor}}$ with the restriction of $\dGH$. Then every nonempty compact metric space is isometrically embeddable into $\M_{\mathrm{Cantor}}$.
\end{corollary}

\begin{proof}
For spaces of positive diameter, all representative spaces occurring in the construction of \cref{thm:main} are homeomorphic to the Cantor space, so the image is contained in $\M_{\mathrm{Cantor}}$. For a one-point space, choose any point of $\M_{\mathrm{Cantor}}$ and take the unique map into that point.
\end{proof}

\section{Finite approximations}

The isometric realization constructed above admits natural finite truncations. We obtain an exact Gromov--Hausdorff distance formula for these finite models, together with uniform approximation and covering estimates.

\begin{definition}\label{def:finite-model}
For $m\ge0$ and $v\in T_m$, define $s_v\in S$ by
\[
 (s_v)_j:=
 \begin{cases}
  v_j,&0\le j<m,\\
  1,&j\ge m.
 \end{cases}
\]
Set $S_m:=\{s_v\mid v\in T_m\}$ and $Z_m:=\pi(S_m)$, and define $Y_f^{(m)}:=Y_{f,S_m}$.
\end{definition}

\begin{lemma}\label{lem:finite-net}
For every $m\ge0$, one has $S_m\subseteq S_{m+1}$, and $Z_m$ is a $\gamma_m$-net of $K$.
\end{lemma}

\begin{proof}
For $v\in T_m$, we have $s_v=s_{v1}$, so $S_m\subseteq S_{m+1}$. Given any $z\in K$, choose $\alpha\in S$ such that $\pi(\alpha)=z$, and put $v:=\alpha|_m$. Both $z$ and $\pi(s_v)$ belong to $K_v$, hence $d_K(z,\pi(s_v))\le\diam K_v\le\gamma_m$. Thus $Z_m$ is a $\gamma_m$-net of $K$.
\end{proof}

\begin{proposition}\label{prop:finite}
For every $m\ge1$ and every $f,g\in\cL_D(K)$,
\begin{equation}\label{eq:finiteexact}
 \dGH(Y_f^{(m)},Y_g^{(m)})
 =\max_{z\in Z_m}|f(z)-g(z)|.
\end{equation}
Consequently,
\begin{equation}\label{eq:finiteerror}
 0\le
 \lVert f-g\rVert_\infty-
 \dGH(Y_f^{(m)},Y_g^{(m)})
 \le2\gamma_m.
\end{equation}
\end{proposition}

\begin{proof}
Set $\delta_m:=\max_{z\in Z_m}|f(z)-g(z)|$.
If $I_m:=\{(x,x)\mid x\in S_m\times\{+,-\}\}$, then $\dis I_m=2\delta_m$ and $\dGH(Y_f^{(m)},Y_g^{(m)})\le\delta_m$.
If $\delta_m=0$, then $\widehat f=\widehat g$ on $S_m$, and therefore $\rho_f|_{(S_m\times\{+,-\})^2}=\rho_g|_{(S_m\times\{+,-\})^2}$. Hence \eqref{eq:finiteexact} follows in this case.

Assume now that $\delta_m>0$. By the definitions of $Z_m=\pi(S_m)$ and $\delta_m$, there exists $\eta\in S_m$ such that $|f(\pi(\eta))-g(\pi(\eta))|=\delta_m$. Interchanging $f$ and $g$ if necessary, assume $f(\pi(\eta))-g(\pi(\eta))=\delta_m$.
Suppose that there exists a correspondence $R\in\Corr(S_m\times\{+,-\},S_m\times\{+,-\})$ between $Y_f^{(m)}$ and $Y_g^{(m)}$ with $\Delta:=\dis R<2\delta_m$.
Since $\delta_m\le D$, we have $\Delta<2D$. Apply \cref{lem:layers} with $E=S_m$ and replace $R$ by the normalized relation $\widetilde R$ from \eqref{eq:normalized-relation}. This replacement leaves the distortion unchanged. From now on, write $R$ again for the normalized correspondence and put
\[
 R_\varepsilon
 :=\{(s,t)\in S_m^2\mid
 ((s,\varepsilon),(t,\varepsilon))\in R\}
 \in\Corr(S_m,S_m)
 \qquad(\varepsilon\in\{+,-\}).
\]

For $0\le k\le m$, set $E_k:=C_{\eta|_k}\cap S_m$.
We prove by induction on $k$ that
\begin{equation}\label{eq:finite-induction}
 R_\varepsilon[E_k]=E_k=R_\varepsilon^{-1}[E_k]
 \qquad(\varepsilon\in\{+,-\}).
\end{equation}
Since $E_0=S_m$, the assertion is true for $k=0$. Assume that \eqref{eq:finite-induction} holds for some $k<m$. Then
\[
 P_k:=R_+\cap(E_k\times E_k),
 \qquad
 Q_k:=R_-\cap(E_k\times E_k)
\]
belong to $\Corr(E_k,E_k)$. Take arbitrary $(s,t)\in P_k$ and $(u,w)\in Q_k$. The points $\pi(s),\pi(t),\pi(u),\pi(w)$ and $\pi(\eta)$ all belong to $K_{\eta|_k}$. Hence, using
$f(\pi(\eta))-g(\pi(\eta))=\delta_m$ and the fact that $f,g$ are $1$-Lipschitz,
\begin{align*}
 \widehat f(s)-\widehat g(t)
 &=f(\pi(s))-f(\pi(\eta))+\delta_m
   +g(\pi(\eta))-g(\pi(t))
 \ge\delta_m-2\gamma_k,\\
 \widehat f(u)-\widehat g(w)
 &=f(\pi(u))-f(\pi(\eta))+\delta_m
   +g(\pi(\eta))-g(\pi(w))
 \ge\delta_m-2\gamma_k.
\end{align*}
Applying the definition of distortion to the two elements
$((s,+),(t,+))$ and $((u,-),(w,-))$ of $R$, and cancelling the common term $50D$ in the cross-layer distances, gives
\[
 \left|
 b(s,u)-b(t,w)
 +\widehat f(s)+\widehat f(u)
 -\widehat g(t)-\widehat g(w)
 \right|
 \le\Delta.
\]
The preceding Lipschitz estimates therefore imply the uniform bound
\[
 \sup_{\substack{(s,t)\in P_k\\(u,w)\in Q_k}}
 \bigl(b(s,u)-b(t,w)\bigr)
 \le\Delta-2\delta_m+4\gamma_k
 <4\gamma_k=\frac{\lambda_k}{2}.
\]
For $k<m$ and $i\in[q_k]$, define $v^{(i)}\in T_m$ by
\[
 v^{(i)}_j:=
 \begin{cases}
  \eta_j,&0\le j<k,\\
  i,&j=k,\\
  1,&k<j<m
 \end{cases}
 \qquad(0\le j<m).
\]
Then
\[
 s_{v^{(i)}}\in E_k\cap C_{(\eta|_k)\,i}.
\]
Thus $E_k$ meets every child cylinder of $C_{\eta|_k}$, and \cref{lem:cylinder-rigidity} applies with $E=E_k$ and $v=\eta|_k$.
The part corresponding to the child label $i=\eta_k$ is $E_k\cap C_{(\eta|_k)\,\eta_k}=E_{k+1}$.
Hence the lemma yields
\[
 P_k[E_{k+1}]=E_{k+1}=P_k^{-1}[E_{k+1}],\qquad
 Q_k[E_{k+1}]=E_{k+1}=Q_k^{-1}[E_{k+1}].
\]
The induction hypothesis at depth $k$ implies
\[
 R_\varepsilon\cap
 \bigl((E_k\times S_m)\cup(S_m\times E_k)\bigr)
 \subseteq E_k\times E_k.
\]
Therefore the preceding equalities for $P_k,Q_k$ give
\[
 R_\varepsilon[E_{k+1}]=E_{k+1}
 =R_\varepsilon^{-1}[E_{k+1}]
 \qquad(\varepsilon\in\{+,-\}),
\]
which proves \eqref{eq:finite-induction} at depth $k+1$.

The map $v\mapsto s_v$ is injective on $T_m$, since the first $m$ coordinates of $s_v$ are exactly $v$. Moreover, $\eta=s_{\eta|_m}$, so at depth $m$, $E_m=C_{\eta|_m}\cap S_m=\{\eta\}$.
Thus $(\eta,\eta)\in R_+\cap R_-$, and the normalized correspondence $R$ contains
\[
 ((\eta,+),(\eta,+)),
 \qquad
 ((\eta,-),(\eta,-)).
\]
The distance discrepancy determined by these two elements is
\[
 \left|
 \rho_f((\eta,+),(\eta,-))
 -\rho_g((\eta,+),(\eta,-))
 \right|
 =2|f(\pi(\eta))-g(\pi(\eta))|
 =2\delta_m,
\]
contradicting the assumption $\Delta<2\delta_m$. Therefore $\dis R\ge2\delta_m$ for every correspondence $R$, and \eqref{eq:GHformula} gives \eqref{eq:finiteexact}.

Finally, let $h:=f-g$. Then $h$ is $2$-Lipschitz. Choose $z\in K$ at which $|h|$ attains its maximum, and use \cref{lem:finite-net} to choose $z_m\in Z_m$ with $d_K(z,z_m)\le\gamma_m$. Then
\[
 \lVert h\rVert_\infty
 =|h(z)|
 \le|h(z_m)|+2d_K(z,z_m)
 \le\delta_m+2\gamma_m.
\]
Together with \eqref{eq:finiteexact}, this proves \eqref{eq:finiteerror}.
\end{proof}

\begin{proposition}\label{prop:uniformcompact}
For every $m\ge1$ and every $f\in\cL_D(K)$,
\begin{equation}\label{eq:uniformfinite}
 \dGH(Y_f^{(m)},Y_f)
 \le6\lambda_m
 =48D\,4^{-m}.
\end{equation}
Moreover, if $\varepsilon\ge6\lambda_m$, then
\begin{equation}\label{eq:covering-number}
 N_{Y_f}(\varepsilon)
 \le2|S_m|
 =2\prod_{j=0}^{m-1}q_j.
\end{equation}
In particular, this covering-number bound is independent of $f\in\cL_D(K)$.
\end{proposition}

\begin{proof}
For any $s\in S$, let $v:=s|_m$. Then $s,s_v\in C_v$, and \eqref{eq:cylbounds} gives
\[
 \rho_f((s,\vartheta),(s_v,\vartheta))
 =2b(s,s_v)
 \le6\lambda_m
 \qquad(\vartheta\in\{+,-\}).
\]
Thus, under the natural inclusion $S_m\times\{+,-\}\subseteq\Omega$, the finite set $S_m\times\{+,-\}$ is a $6\lambda_m$-net in $Y_f$. Since $Y_f^{(m)}$ carries the restricted metric, this inclusion is isometric, and therefore
\[
 \dGH(Y_f^{(m)},Y_f)
 \le d_\mathrm{H}^{Y_f}(S_m\times\{+,-\},\Omega)
 \le6\lambda_m,
\]
which proves \eqref{eq:uniformfinite}. If $\varepsilon\ge6\lambda_m$, then the same finite set is also an $\varepsilon$-net of $Y_f$. Since the map $v\mapsto s_v$ is injective on $T_m$ by definition, $|S_m|=|T_m|=\prod_{j=0}^{m-1}q_j$. Therefore $N_{Y_f}(\varepsilon)\le|S_m\times\{+,-\}|=2|S_m|$, which is \eqref{eq:covering-number}.
\end{proof}

\section*{Acknowledgments}
The author would like to thank his advisor, Professor Takashi Shioya, for many helpful suggestions and guidance. The author is also grateful to Yoshito Ishiki for valuable comments on this manuscript.
The author used the GPT-5.6 and GPT-6 models to prepare this manuscript.
The author reviewed and revised the mathematical content and takes full responsibility for the final manuscript.

\bigskip
\noindent\textsc{Mathematical Institute, Tohoku University}\\
Sendai 980-8578, Japan\\
\textit{Email address:} \texttt{fukano.ryo.t5@dc.tohoku.ac.jp}

\end{document}